\documentclass[12pt,a4paper,reqno]{amsart}
\usepackage[utf8]{inputenc}
\usepackage{enumitem}
\usepackage{amsmath}
\usepackage{amsfonts}
\usepackage{amssymb}
\usepackage[all]{xy}
\usepackage{xfrac}
\usepackage[yyyymmdd,hhmmss]{datetime}
\usepackage{centernot}
\usepackage[english]{babel}
\usepackage{dsfont}
\usepackage{bbold}
\usepackage{soul}
\usepackage{lmodern}
\usepackage{lipsum}

\newtheorem{theorem}{Theorem}[section]

\newtheorem{lemma}[theorem]{Lemma}
\newtheorem{definition}{Definition}[section]
\newtheorem{proposition}[theorem]{Proposition}

\newcommand{\ov}{\overline}

\newcommand{\p}{p }

\usepackage{mathtools}
\newcommand{\eq}[1]{\stackrel{\mathclap{\normalfont\mbox{{\tiny #1}}}}{=}}
\newcommand{\impli}[1]{\stackrel{\mathclap{\normalfont\mbox{{\tiny #1}}}}{\Rightarrow}}

\author{J. P. Fatelo and N. Martins-Ferreira}

\address{School of Technology and Management, Centre for Rapid and Sustainable Product Development - CDRSP, Polytechnic Institute of Leiria, P-2411-901 Leiria, Portugal.}

\email{martins.ferreira@ipleiria.pt}

\title[]{Reconstructing Classical Algebras via Ternary Operations}

								               

\thanks{This research work was supported by the Portuguese Foundation for Science and Technology FCT/MCTES (PIDDAC) through the following Projects: Associate Laboratory ARISE LA/P/0112/2020; UIDP/04044/2020; UIDB/04044/2020; PAMI - ROTEIRO/0328/2013 (N° 022158); MATIS (CENTRO-01-0145-FEDER-000014 - 3362); DOI: 10.54499/UIDB/04044/2020; Generative.Thermodynamic; FruitPV; by CDRSP and ESTG from the Polytechnic Institute of Leiria.
}
 
\begin{document}

\begin{abstract}
Although algebraic structures are frequently analyzed using unary and binary operations, they can also be effectively defined and unified through ternary operations. In this context, we introduce structures that contain two constants and a ternary operation. We demonstrate that these structures are isomorphic to various significant algebraic systems, including Boolean algebras, de Morgan algebras, MV-algebras, and (near) rings of characteristic two. Our work highlights the versatility of ternary operations in describing and connecting diverse algebraic structures.
\end{abstract}

\maketitle

\section{Introduction}
Ternary Boolean algebras \cite{Grau} were introduced by Grau in 1947 to axiomatize Boolean algebras by means of the ternary operation $xy\circ yz\circ zx$ (see also~\cite{Pad}). In the same year this operation was used independently by Birkhoff and Kiss~\cite{Birk} to characterize distributive lattices. Both approaches are particular cases of median algebras \cite{Bandelt, Isbell, Sholander}.  
Although the set of axioms  is distinct in each case, complete commutativity is a common feature. 
In 1948, Church \cite{ChurchPM} shows that it is possible to axiomatize Boolean algebras in terms of the conditioned disjunction $\bar{y}x\circ yz$ which is not completely commutative (see also \cite{hoare}). 

In this article, we explore the axiomatization of Boolean algebras, de Morgan algebras, MV-algebras, and rings or near rings of characteristic two using ternary structures. By a ternary structure we mean an algebraic system consisting of a set $A$ with two constants, $0$ and $1$, and a ternary operation $p(x,y,z)\in A$. This ternary operation gives rise to derived unary and binary operations, and each formula specifying $p$ in terms of these operations corresponds to a new axiom, leading to subvarieties of the original structure.
Interestingly, each classical structure considered here has a characteristic expression that determines the ternary operation $p$ using its derived binary operations. The purpose of this paper is to examine these expressions and their implications for the unification of classical algebras within the framework of ternary structures.

For example, within the set of axioms \ref{C1}--\ref{C3}, introduced in Section \ref{section:preliminaries}, Boolean algebras form a subvariety if and only if $p(x,y,z)$ is interpreted as Church's conditioned disjunction, as proved in Section \ref{section:Boolean}.
Section \ref{section:DM} gives the formula for $p$ that turns the \ref{C1}--\ref{C3} structure into a de Morgan algebra, while Section \ref{section:rings} discusses the cases of rings and near rings of characteristic 2. In Section \ref{section:MV}, it is shown how to modify the axioms
\ref{C1}--\ref{C3} so that MV-algebras can be characterized with a ternary structure and the corresponding ternary operation is given. Finally, conclusions are draw in Section \ref{section:conclusion} where a table is presented with a comparison between the cases considered here.

Further observations and some more details are presented in the preprints~\cite{preprint-TBA} and \cite{preprint-TBA2}.
\section{Preliminaries}\label{section:preliminaries}   

Let us begin by introducing the principal notations used for the unary and binary operations derived from a general structure with one ternary operation $p$ and two constants $0$ and $1$.
\begin{definition}\label{definition:1}
Let $(A,p,0,1)$ be a system consisting of a set $A$, together with a ternary operation $p$ and two constants $0,1\in A$. From $p$, the following operations are defined:
\begin{eqnarray}
\label{def:bar}\bar a      &=&p(1,a,0)\\
\label{def:cdot}a\cdot b   &=&p(0,a,b)\\
\label{def:circ}a\circ b   &=&p(a,b,1)\\
\label{def:wedge}a\wedge b &=&p(b,\bar a,0)\\
\label{def:vee}a\vee b     &=&p(1,\bar b,a)\\
\label{def:plus}a+b        &=&p(a,b,\bar a).
\end{eqnarray}
\end{definition}

Next, we present an algebraic structure $(A,p,0,1)$ satisfying four axioms. Axiom~\ref{C4} alone defines a Church algebra~\cite{Salibra2,Salibra1} while a Menger algebra of rank 2 (see e.g. \cite{menger}) uses axiom \ref{C3}.
The axioms \ref{C1}, \ref{C4} and \ref{C3} have been used to define  proposition algebras~\cite{Ponse} while axioms \ref{C4}, \ref{C2} and \ref{C3}, among others, have been used to study spaces with geodesic paths~\cite{mobi,affine,mobi2sphere}. Lemma \ref{lemma:1} below displays basic consequences of \ref{C1}--\ref{C3}.
\begin{lemma}\label{lemma:1}
Let $(A,p,0,1)$ be a system consisting of a set $A$, together with a ternary operation $p$ and two constants $0,1\in A$ satisfying:
\begin{enumerate}[label={\bf (T\arabic*)}]
\item\label{C1}\label{B1}\label{smA2} $p(0,a,1)=a$
\item\label{C4}\label{B4}\label{smA4}\label{smA5} $p(a,0,b)=a=p(b,1,a)$
\item\label{C2}\label{B2}\label{smA3} $p(a,b,a)=a$
\item\label{C3}\label{B3}\label{smA7} $p(a,p(b_1,b_2,b_3),c)=
p(p(a,b_1,c),b_2,p(a,b_3,c))$.
\end{enumerate}
Then, the following properties hold:
\begin{eqnarray}
\label{L1}\bar{1}=0 &,& \bar{0}=1\\
\label{L2} \overline{\overline{a}}&=&a\\
\label{L3} p(c,b,a)&=&p(a,\bar{b},c)\\
\label{L4} \overline{p(a,b,c)}&=&p(\overline{a},b,\overline{c})\\
\label{L5} \overline{p(a,b,c)}&=&p(\overline{c},\overline{b},\overline{a})\\
\label{L=} a\cdot b=a\wedge b&,& a\circ b=a\vee b\\
\label{L6} \overline{a\cdot b}=\overline{b}\circ \overline{a}&,&\overline{a\circ b}=\overline{b}\cdot \overline{a}\\
\label{L7}\label{L8} (A,\cdot,1)\ \textrm{and}\ (A,\circ,0)\ &\textrm{are}&\textrm{monoids}\\
\label{L9} a\cdot 0=&0&=0\cdot a\\
\label{L10} a\circ 1=&1&=1\circ a\\
\label{monoid-plus}(A,+,0)&\textrm{is}& \textrm{a monoid}\\
\label{L11} a+1=&\bar a&=1+a
\end{eqnarray}
\end{lemma}
\begin{proof}
In each step of the proof, the needed property when required is written above the corresponding equality.
$$\begin{array}{l}
\bar{1}=p(1,1,0)\ \eq{\ref{C4}}\ 0,\quad \bar{0}=p(1,0,0)\ \eq{\ref{C4}}\ 1\\[5pt]
\overline{\overline{a}}=p(1,p(1,a,0),0)\ \eq{\ref{C3}}\ p(p(1,1,0),a,p(1,0,0))\ \eq{\ref{C4}}\ p(0,a,1)\ \eq{\ref{C1}}\ a\\[5pt]
p(a,\overline{b},c)=p(a,p(1,b,0),c)\ \eq{\ref{C3}}\ p(p(a,1,c),b,p(a,0,c))\ \eq{\ref{C4}}\ p(c,b,a)\\[5pt]
\overline{p(a,b,c)}=p(1,p(a,b,c),0)\ \eq{\ref{C3}}\ p(p(1,a,0),b,p(1,c,0))=p(\overline{a},b,\overline{c}).
\end{array}$$
Property (\ref{L5}) is just a combination of (\ref{L3}) and (\ref{L4}), whereas (\ref{L=}) is just a particular case of  (\ref{L3}). Next is the proof of Properties (\ref{L6}):
$$\begin{array}{l}
\overline{a\cdot b}=\overline{p(0,a,b)}\ \eq{(\ref{L5})}\ p(\bar{b},\bar{a},\bar{0})\ \eq{(\ref{L1})}\ p(\bar{b},\bar{a},1)
=\bar{b}\circ\bar{a} \\[5pt]
\overline{a\circ b}=\overline{p(a,b,1)}\ \eq{(\ref{L5})}\ p(\bar{1},\bar{b},\bar{a})\ \eq{(\ref{L1})}\ p(0,\bar{b},\bar{a})
=\bar{b}\cdot\bar{a}.
\end{array}$$
With respect to (\ref{L7}), we have associativity
$$\begin{array}{rcccl}
(a\cdot b)\cdot c&=&p(0,p(0,a,b),c)&\eq{\ref{C3}}& p(p(0,0,c),a,p(0,b,c))\\[5pt]
                   &\eq{\ref{C4}}&p(0,a,p(0,b,c))&=&a\cdot(b\cdot c)\\[5pt]
a\circ(b\circ c)&=&p(a,p(b,c,1),1)&\eq{\ref{C3}}& p(p(a,b,1),c,p(a,1,1))\\[5pt]
              &\eq{\ref{C4}}&p(p(a,b,1),c,1)&=&(a\circ b)\circ c,
\end{array}$$
and identities
$$\begin{array}{l}
a\cdot 1=p(0,a,1)\ \eq{\ref{C1}}\ a,\quad 1\cdot a=p(0,1,a)\ \eq{\ref{C4}}\ a\\[5pt]
a\circ 0=p(a,0,1)\ \eq{\ref{C4}}\ a,\quad 0\circ a=p(0,a,1)\ \eq{\ref{C1}}\ a.
\end{array}$$
For properties (\ref{L9}) and (\ref{L10}), the proof is:
$$\begin{array}{l}
a\cdot 0=p(0,a,0)\ \eq{\ref{C2}}\ 0,\quad 0\cdot a=p(0,0,a)\ \eq{\ref{C4}}\ 0\\[5pt]
a\circ 1=p(a,1,1)\ \eq{\ref{C4}}\ 1,\quad 1\circ a=p(1,a,1)\ \eq{\ref{C2}}\ 1.
\end{array}$$
The structure $(A,+,0)$ is a monoid since:
$$a+0=p(a,0,\bar a)\ \eq{\ref{C4}}\ a,\quad
0+a=p(0,a,1)\ \eq{\ref{C1}}\ a$$
\begin{eqnarray*}
(a+b)+c&=&p(p(a,b,\bar a),c,\overline{p(a,b,\bar a})\ \eq{(\ref{L5})}\ p(p(a,b,\bar a),c,p(a,\bar b,\bar a))\\
			 &\eq{\ref{C3}}& p(a,p(b,c,\bar b),\bar a)=a+(b+c).
\end{eqnarray*}
Furthermore $a+1=p(a,1,\bar a)\ \eq{\ref{C2}}\ \bar a$ and $1+a=p(1,a,0)=\bar a$ which proves (\ref{L11}). In particular, $1+1=0$.
\end{proof}

The following lemmas investigate the consequences of some classic extra conditions on the binary operations besides Axioms \ref{C1} to~\ref{C3}. Note that the de Morgan's laws (\ref{L6}) imply a duality between~$\cdot$ and $\circ$. In particular, $\cdot$ is idempotent if and only if $\circ$ is idempotent too. 
\begin{lemma}\label{lemma:idem}
Let $(A,p,0,1)$ be a system verifying conditions \ref{C1} to~\ref{C3}. If the operation $\circ$ is idempotent then we have the following absorption rules:
\begin{equation} \label{abs}
a\circ(b\cdot a)=a\quad \textrm{and}\quad ( a \circ b)\cdot a=a.
\end{equation}
\end{lemma}
\begin{proof}
Idempotency means that
\begin{equation}\label{idemDM}
p(0,a,a)=a\cdot a=a\quad \textrm{and}\quad p(a,a,1)=a\circ a=a,
\end{equation}
and consequently:
$$\begin{array}{rcccl}
a\circ(b\cdot a)&=&\p(a,\p(0,b,a),1)&\eq{\ref{smA7}}&\p(\p(a,0,1),b,\p(a,a,1))\\[5pt]
           &\eq{\ref{smA4}}&\p(a,b,\p(a,a,1))&\eq{(\ref{idemDM})}&\p(a,b,a)\ \eq{\ref{smA3}}\ a.
\end{array}$$
The second equality in (\ref{abs}) follows by application of (\ref{L6}).
\end{proof}

When commutativity is added, distributivity is obtained.
\begin{lemma}\label{lemma:2}
Let $(A,p,0,1)$ be a system verifying conditions \ref{C1} to~\ref{C3}. If the operations $\cdot$ and $\circ$ are commutative and idempotent then they distribute over each other.
\end{lemma}
\begin{proof}
When $\circ$ is commutative, the first absorption rule in (\ref{abs}) may be written as:
\begin{equation}\label{absorption}
(b\cdot a)\circ a=\p(\p(0,b,a),a,1)=a.
\end{equation}
Then, we have:
\begin{eqnarray*}
(b\cdot a)\circ (c\cdot a)&=&\p(\p(0,b,a),\p(0,c,a),1)\\
                    &\eq{\ref{smA7}}&\p(\p(\p(0,b,a),0,1),c,\p(\p(0,b,a),a,1))\\
                    &\eq{\ref{smA4}}&\p(\p(0,b,a),c,\p(\p(0,b,a),a,1))\\
										&\eq{(\ref{absorption})}&\p(\p(0,b,a),c,a)\\
										&\eq{\ref{smA5}}&\p(\p(0,b,a),c,\p(0,1,a))\\
										&\eq{\ref{smA7}}&\p(0,\p(b,c,1),a)=(b\circ c)\cdot a.
\end{eqnarray*}
Similarly, using
$(a\circ b)\cdot a=a$, we get $(a\circ b)\cdot(a\circ c)=a\circ(b\cdot c).$
\end{proof}

\begin{lemma}\label{lemma:a+a}
Let $(A,p,0,1)$ be a system verifying conditions \ref{C1} to~\ref{C3}. If $a+a=0$ then the operation $+$ is commutative and the operation $\cdot$ is right distributive over $+$:
\begin{equation}
a+a=0\Rightarrow a+b=b+a\ \textrm{and}\ (a+b)\cdot c=a\cdot c+b\cdot c.
\end{equation}
\end{lemma}
\begin{proof}
$a+a=0$ implies, using (\ref{monoid-plus}), $(a+b)+(b+a)=0$ and therefore $a+b=b+a$. Right distributivity of $\cdot$ over $+$ can be proved as follows:
\begin{eqnarray*}
(a\cdot c)+(b\cdot c)&\quad=\quad&p(a\cdot c,p(0,b,c),\overline{a\cdot c})\\
                   &\eq{\ref{C3},\ref{C4}}&p(a\cdot c,b,p(a\cdot c,c,\overline{a\cdot c}))\\
									 &=&p(a\cdot c,b,(a\cdot c)+c)\\
									 &=&p(a\cdot c,b,c+(a\cdot c))\\
									 &=&p(a\cdot c,b,p(c,p(0,a,c),\bar c))\\
									 &\eq{\ref{C3},\ref{C4}}&p(a\cdot c,b,p(c,a,p(c,c,\bar c))\\
									 &=&p(p(0,a,c),b,p(c,a,0))\\
									 &\eq{(\ref{L3})}&p(p(0,a,c),b,p(0,\bar a,c))\\
                   &\eq{\ref{C3}}&p(0,p(a,b,\bar a),c)=(a+b)\cdot c.
\end{eqnarray*} 
\end{proof}
\noindent It is worth noting that left distributivity of $\cdot$ over $+$ is not guaranteed. On the dual side, we have left distributivity of $\circ$ over the binary operation $a*b=p(\bar a,b,a)=\bar a+b$ when $a*a=1$ and no right distributivity guaranteed.

In addition to the structure defined by conditions \ref{C1}--\ref{C3}, in the following lemma \mbox{$\bar a=p(1,a,0)$} is assumed to be the Boolean complement of $a\in A$. When this is the case $\cdot$ and~$\circ$ are idempotent.
\begin{lemma}\label{lemma:complement}
Let $(A,p,0,1)$ be a system verifying conditions \ref{C1} to~\ref{C3}. If every $a\in A$ verifies the relations
\begin{equation}\label{complement}
\bar a\cdot a=p(0,\bar a,a)=0\quad \textrm{and}\quad \bar a\circ a=p(\bar a,a,1)=1,
\end{equation}
then idempotency holds:
\begin{equation}\label{idem}
p(0,a,a)=a\cdot a=a\quad \textrm{and}\quad p(a,a,1)=a\circ a=a.
\end{equation}
\end{lemma}
\begin{proof}
$$\begin{array}{rcccl}
a\cdot a&=&p(0,a,a)\quad \eq{(\ref{complement}),\ref{C4}}\quad p(p(0,\bar a,a),a,p(0,1,a))\\[5pt]
         &\eq{\ref{C3}}&\hspace{-23pt}p(0,p(\bar a,a,1),a)\ \eq{(\ref{complement})}\ p(0,1,a)\ \eq{\ref{C4}}\ a
\end{array}$$
Idempotency of $\circ$ is obtained similarly or using (\ref{L6}).
\end{proof}

\section{Boolean algebra}\label{section:Boolean}   

The next proposition shows how  the structure of axioms \ref{C1}--\ref{C3} can be turned into a Boolean ring. Recall that the notation $a+b$ is being used for $p(a,b,\bar a)$.
\begin{proposition}\label{theorem:ring}
If $(A,p,0,1)$ verifies conditions \ref{C1} to~\ref{C3} and if
\begin{equation}\label{T3-2}
p(0,a,b)=p(a,a,b)
\end{equation}
then:
\begin{equation}
\label{ring}(A,+,\cdot,0,1)\ \textrm{is a Boolean ring}.  
\end{equation}
\end{proposition}
\begin{proof}
Condition (\ref{T3-2}) implies Boolean complements and $a+a=0$:
\begin{eqnarray}
a\cdot\bar a&=&p(0,a,\bar a)\ \eq{(\ref{L3})}\ p(\bar a,\bar a,0)\ \eq{(\ref{T3-2})}\ p(0,\bar a,0)\ \eq{\ref{C2}}\ 0,\nonumber\\
\label{zero-plus}a+a&=&p(a,a,\bar a)\ \eq{(\ref{T3-2})}\ p(0,a,\bar a)=a\cdot \bar a=0.
\end{eqnarray}
Now, this result and Lemma \ref{lemma:a+a} imply that $+$ is commutative
\begin{equation}
\label{commut-plus}p(a,b,\bar a)=p(b,a,\bar b)
\end{equation}
and $\cdot$ distribute on the right over $+$
\begin{equation}
(a\cdot c)+(b\cdot c)=(a+b)\cdot c.\label{RD}
\end{equation} 
In addition, the following properties hold:
\begin{equation}
\label{LD}a\cdot(a+b)=a+(a\cdot b),\quad b\cdot(a+b)=(b\cdot a)+b.
\end{equation}
Indeed:
\begin{eqnarray*}
a\cdot(a+b)&=&p(0,a,p(a,b,\bar a)\ \eq{(\ref{zero-plus})}\ p(p(a,a,\bar a),a,p(a,b,\bar a))\\
            &\eq{\ref{C3}}&p(a,p(a,a,b),\bar a)\ \eq{(\ref{T3-2})}\ p(a,a\cdot b,\bar a)=a+(a\cdot b).
\end{eqnarray*}
The second relation in (\ref{LD}) is a consequence of the commutativity of~$+$.
We can now prove that $\cdot$ is commutative. From Lemma~\ref{lemma:complement} we already know that, under the hypothesis of Proposition~\ref{theorem:ring}, $\cdot$ is idempotent and consequently 
\begin{eqnarray*}
(a+b)\cdot (a+b)=a+b &\impli{(\ref{RD})}&\big(a\cdot(a+b)\big)+\big(b\cdot(a+b)\big)=a+b\\
                      &\impli{(\ref{LD})}&\big(a+(a\cdot b)\big)+\big((b\cdot a)+b\big)=a+b\\
                      &\impli{(\ref{monoid-plus}),(\ref{commut-plus})}&(a\cdot b)+(b\cdot a)+a+b=a+b\\
                      &\impli{(\ref{monoid-plus}),(\ref{zero-plus})}&(a\cdot b)+(b\cdot a)=0\\
											&\impli{(\ref{monoid-plus}),(\ref{zero-plus})}& a\cdot b=b\cdot a.
\end{eqnarray*}
\end{proof}

The following theorem is a refinement of Grau's  ternary Boolean algebra in the sense that it uses Church's  operation and a systematisation of Hoare's axioms considered in \cite{hoare}.

\begin{theorem}\label{thm:1}
Suppose that $(A,p,0,1)$ satisfies axioms \ref{C1} to~\ref{C3}. For 
$$\bar{a}=p(1,a,0),\ a\cdot b=p(0,a,b),\ a\circ b=p(a,b,1)\ \,\textrm{and}\ \,a+b=p(a,b,\bar a),$$
the following conditions are equivalent:
\begin{enumerate}[label={\rm(\roman*)}]
\item\label{T5} $(A,+,\cdot,0,1)$ is a Boolean ring
\item\label{T1} $(A,\circ,\cdot,\bar{()},0,1)$ is a Boolean algebra
\item\label{T2} $p(a,b,c)=(\bar{b}\cdot a)\circ (b\cdot c)$
\item\label{T3} $p(a,a,b)=a\cdot b$
\item\label{T4} $p(a,b,b)=a\circ b$.
\end{enumerate}
\end{theorem}
\begin{proof}
The proof proceeds as follows: \ref{T1} $\Rightarrow$ \ref{T2} $\Rightarrow$ (\ref{T3} $\Leftrightarrow$ \ref{T4}) $\Rightarrow$ \ref{T5} $\Rightarrow$ \ref{T1}.
We begin by proving that if $(A,p,0,1)$ is a system verifying the hypothesis of Theorem \ref{thm:1} then \ref{T1} implies \ref{T2}. It is well known (see e.g. \cite{Birk-livro, NMF2012}) that, in a distributive lattice, if $x\cdot a=x'\cdot a$ and $a\circ x=a\circ x'$ for some given element $a$ in the lattice then $x=x'$. We show here that if $(A,p,0,1)$ verifies \ref{C1} to~\ref{C3} and $(A,\circ,\cdot,\bar{()},0,1)$ is a Boolean algebra then
$$\left\{\begin{array}{rcl}
p(a,b,c)\cdot c&=&((\bar{b}\cdot a)\circ(b\cdot c))\cdot c\\
c\circ p(a,b,c)&=&c\circ ((\bar{b}\cdot a)\circ(b\cdot c))
\end{array}\right.,$$
which proves $\ref{T2}$. Indeed:
$$\begin{array}{rcccl}
p(a,b,c)\cdot c&=&p(0,p(a,b,c),c)&\eq{\ref{smA7}}&p(p(0,a,c),b,p(0,c,c))\\[5pt]
                &\eq{(\ref{idem})}&p(p(0,a,c),b,c)&\eq{\ref{C4}}&p(p(0,a,c),b,p(0,1,c))\\[5pt]
								&\eq{\ref{smA7}}&p(0,p(a,b,1),c)&=&(a\circ b)\cdot c\\[5pt]
								&=&((\overline{b}\cdot a)\circ b)\cdot c
								&=&((\overline{b}\cdot a)\circ(b\cdot c))\cdot c
\end{array}$$
$$\begin{array}{rcccl}
c\circ p(a,b,c)&=&p(c,p(a,b,c),1)&\eq{\ref{smA7}}&p(p(c,a,1)),b,p(c,c,1))\\[5pt]
							&\eq{(\ref{idem})}&p(p(c,a,1),b,c)&\eq{\ref{C4}}&p(p(c,a,1),b,p(c,0,1))\\[5pt]
							&\eq{\ref{smA7}}&p(c,p(a,b,0),1)&\eq{(\ref{L3})}&p(c,p(0,\overline{b},a),1)\\[5pt]
							&=&c\circ (\overline{b}\cdot a)
							&=&c\circ((\overline{b}\cdot a)\circ (b\cdot c)).
\end{array}$$
Next, it is shown that condition $\ref{T2}$ implies condition $\ref{T3}$. Indeed, when $\ref{T2}$ is true, we have:
$$p(1,a,1)=(\bar{a}\cdot1)\circ(a\cdot 1)$$
which means, using \ref{C2} and (\ref{L7}), that $1=\bar{a}\circ a$ and, by duality, that $\bar a\cdot a=0$. Therefore $p(a,a,b)=(\bar a\cdot a)\circ(a\cdot b)=a\cdot b$. Conditions $\ref{T3}$ and $\ref{T4}$ are equivalent by duality (\ref{L6}). Proposition \ref{theorem:ring} proves that $\ref{T3}$ implies $\ref{T5}$. It remains to prove \ref{T5} $\Rightarrow$ \ref{T1}, that is, if $(A,+,\cdot,0,1)$ is a Boolean ring then $(A,\circ,\cdot,\overline{()},0,1)$ is a Boolean algebra with $\bar a$ defined as $a+1$ and $a\circ b$ defined as $a+b+a\cdot b$. Indeed, firstly we have:
\begin{equation}\label{a-plus-1}
a+1=p(a,1,\bar a)\ \eq{\ref{C4}}\ \bar a
\end{equation}
and consequently
\begin{equation}\label{a-plus-bar}
a+(a+1)=a+\bar a\quad\ \impli{(\ref{monoid-plus}),(\ref{zero-plus})}\quad\ a+\bar a=1.
\end{equation}
Secondly, we have
\begin{equation}\label{a-plus}
a+b+a\cdot b\ \eq{(\ref{ring})}\ a+(b\cdot(a+1))\ \eq{(\ref{a-plus-1})}\ a+(b\cdot\bar a)
\end{equation}
\begin{eqnarray*}
\Rightarrow a+b+a\cdot b&\eq{(\ref{a-plus})}& p(a,p(0,b,\bar a),\bar a)\quad\ \eq{\ref{C3},\ref{C4}}\quad\ p(a,b,p(a,\bar a,\bar a))\\
             &=& p(a,b,a+\bar a)\ \eq{(\ref{a-plus-bar})}\ p(a,b,1)=a\circ b.
\end{eqnarray*}
\end{proof}
It is straightforward to prove that the category of Boolean algebras is isomorphic to the category of structures $(X,p,0,1)$ of type (3,0,0) verifying \ref{C1}--\ref{C3} together with $p(a,b,c)=p(p(0,\bar b, a),p(0,b,c),1)$ or any other equivalent formulation presented in Theorem \ref{thm:1}. A detailed proof of this result is given in the next section for the more general case of de Morgan algebras.

\section{De Morgan algebra}\label{section:DM}   
A de Morgan algebra is a structure $(A,\cdot,\circ,\bar{()},0,1)$ consisting of a bounded distributive lattice $(A,\cdot,\circ,0,1)$ together with an involution~$\overline{()}$ verifying $\overline{a\cdot b}=\bar b\circ \bar a$ (or $\overline{a\circ b}=\bar b\cdot \bar a$).
Simple examples of de Morgan algebras are the sets of divisors of any given positive integer $n$ with $gdc$ as $\cdot$, $lmc$ as $\circ$ and $\bar a=n/a$. Multiple valued-logic \cite{Belnap, Kalman, Font} are also examples of de Morgan algebras. In this section, a characterization  of de Morgan algebras in terms of a ternary structure is given.
\begin{theorem}\label{thm:DM}
Let $(A,p,0,1)$ be a system consisting of a set $A$, a ternary operation $p$ and two constants $0,1\in A$ satisfying the conditions \ref{C1} to \ref{C3}.
For $a\cdot b=p(0,a,b)$, $a\circ b=p(a,b,1)$ and $\bar{a}=p(1,a,0)$, the following conditions are equivalent:
\begin{enumerate}[label={\rm(\roman*)}]
\item\label{DMT1} The system $(A,\circ,\cdot,\bar{()},0,1)$ is a de Morgan algebra
\item\label{DMT-dist} The system $(A,\circ,\cdot)$ is a distributive lattice
\item\label{DMT-0} $(A,\circ)$ is a join-semilattice
\item\label{DMT-1} $(A,\cdot)$ is a meet-semilattice
\item\label{DMT-ou} $(A,\circ)$ is an idempotent and commutative magma
\item\label{DMT3} $(A,\cdot)$ is an idempotent and commutative magma
\item\label{DMT2} $p(a,b,c)=(\bar{b}\cdot a)\circ (a\cdot c)\circ (b\cdot c)$
\item\label{DMT2-2} $p(a,b,c)=(\bar{b}\circ c)\cdot (b\circ a)\cdot (a\circ c)$
\end{enumerate}
\end{theorem}
\begin{proof}
We begin by proving that if $(A,p,0,1)$ is a system verifying the conditions \ref{C1}--\ref{C3} then $\ref{DMT3}$ implies $\ref{DMT1}$. Indeed, $(A,\cdot,0)$ and $(A,\circ,1)$ are monoids as demonstrated in Lemma \ref{lemma:1}. Commutativity and idempotency of $\cdot$ are precisely what $\ref{DMT3}$ states. Commutativity and idempotency of $\circ$ follows by duality. Distributivity of $\cdot$ and $\circ$ over each other is guaranteed by Lemma \ref{lemma:2}. As shown in Lemma \ref{lemma:1}, the unary operation $\bar{()}$ is an involution and verifies the de~Morgan laws~(\ref{L6}), which concludes the proof $\ref{DMT3}\Rightarrow\ref{DMT1}$.

We will now prove that the ternary operation $\ref{DMT2}$ is the only one compatible with $\ref{DMT1}$.
It is well known (see e.g. \cite{Birk-livro, NMF2012}) that, in a distributive lattice, if $x\cdot a=x'\cdot a$ and $a\circ x=a\circ x'$ for some given element $a$ in the lattice then $x=x'$. We show here that if $(A,p,0,1)$ verifies \ref{C1}--\ref{C3} and $(A,\circ,\cdot,\bar{()},0,1)$ is a de Morgan algebra then
$$\left\{\begin{array}{rcl}
p(a,b,c)\cdot c&=&((\bar{b}\cdot a)\circ(a\cdot c)\circ(b\cdot c))\cdot c\\
c\circ p(a,b,c)&=&c\circ ((\bar{b}\cdot a)\circ(a\cdot c)\circ(b\cdot c))
\end{array}\right.,$$
which proves $\ref{DMT1}\to\ref{DMT2}$. Indeed:
$$\begin{array}{rcccl}
p(a,b,c)\cdot c&=&p(0,p(a,b,c),c)&\eq{\ref{smA7}}&p(p(0,a,c),b,p(0,c,c))\\[5pt]
                &\eq{(\ref{idem})}&p(p(0,a,c),b,c)&\eq{\ref{C4}}&p(p(0,a,c),b,p(0,1,c))\\[5pt]
								&\eq{\ref{smA7}}&p(0,p(a,b,1),c)&=&(a\circ b)\cdot c;
\end{array}$$
$$((\bar{b}\cdot a)\circ(a\cdot c)\circ(b\cdot c))\cdot c= ((\bar{b}\cdot a)\circ a\circ b)\cdot c=(a\circ b)\cdot c;$$
$$\begin{array}{rcccl}
c\circ p(a,b,c)&=&p(c,p(a,b,c),1)&\eq{\ref{smA7}}&p(p(c,a,1)),b,p(c,c,1))\\[5pt]
							&\eq{(\ref{idem})}&p(p(c,a,1),b,c)&\eq{\ref{C4}}&p(p(c,a,1),b,p(c,0,1))\\[5pt]
							&\eq{\ref{smA7}}&p(c,p(a,b,0),1)&\eq{(\ref{L3})}&p(c,p(0,\overline{b},a),1)\\[5pt]
							&=&\hspace{-20pt}c\circ (\overline{b}\cdot a);&&
\end{array}$$
$$c\circ(\overline{b}\cdot a)\circ(a\cdot c)\circ (b\cdot c)=c\circ ((a\circ b)\cdot c)\circ(\overline{b}\cdot a)=c\circ (\overline{b}\cdot a).$$

We now turn to the proof that condition $\ref{DMT2}$ implies condition $\ref{DMT3}$ in the context of Theorem \ref{thm:DM}. When $\ref{DMT2}$ is true, we have in particular that
\mbox{$p(b,1,a)=(\bar 1\cdot b)\circ(b\cdot a)\circ(1\cdot a)$}
which means, using \ref{C4}, (\ref{L1}), (\ref{L7}), (\ref{L8}) and (\ref{L9}), that 
\begin{equation}\label{abs1}
a=(b\cdot a)\circ a.
\end{equation} 
By duality, we also have the absorption rule 
\begin{equation}\label{abs3}
a\cdot(a\circ b)=a.
\end{equation}
Idempotency of $\cdot$ and $\circ$ follows as particular cases:
\begin{equation}\label{idempotency}
a\eq{(\ref{abs1})}(1\cdot a)\circ a\eq{(\ref{L7})}a\circ a,\qquad a\eq{(\ref{abs3})}a\cdot (a\circ 0)\eq{(\ref{L8})}a\cdot a.
\end{equation}
When $\ref{DMT2}$ is true, we also have \mbox{$p(a,b,a)=(\bar b\cdot a)\circ(a\cdot a)\circ(b\cdot a)$}
which means, using \ref{C2}, (\ref{idempotency}) and (\ref{abs1}), that 
\begin{equation}\label{abs2}
a=a\circ (b\cdot a).
\end{equation} 
Now, commutativity of $\circ$ can be proved as follows:
$$a\circ b=p(a,b,1)\eq{(\ref{L3})}p(1,\bar b,a)\ \ \eq{$\ref{T2}$,(\ref{L2})}\ \ (b\cdot 1)\circ (1\cdot a)\circ(\bar b\cdot a)\ \ \eq{(\ref{L7}),(\ref{abs2})}\ \ b\circ a.$$
The commutativity of $\cdot$ follows by duality. This proves $\ref{DMT2}\Rightarrow\ref{DMT3}$ and concludes the proof of Teorem \ref{thm:DM} because the other equivalences are trivially verified.
\end{proof}

We show now that the category of de Morgan algebras is isomorphic to the category of systems $(A,p,0,1)$ satisfying conditions \ref{C1}--\ref{C3} when the operations $\cdot$ or $\circ$ are commutative and idempotent. 
\begin{theorem}\label{thm:iso}
Let $(A,p,0,1)$ be a system consisting of a set $A$, together with a ternary operation $p$ and two constants $0,1\in A$ satisfying conditions 
\ref{C1} to \ref{C3} and
\begin{enumerate}[label={\bf (T\arabic*)}]
\setcounter{enumi}{4}
\item\label{C5-DM}$p(0,a,b)=p(0,b,a)$ and $p(0,a,a)=a$.
\end{enumerate}
The category of such systems is isomorphic to the category of de Morgan algebras.
\end{theorem}
\begin{proof}
Consider a system $(A,p,0,1)$ verifying \ref{C1}--\ref{C5-DM} and define 
\begin{equation}\label{e-ou-bar}
a\cdot b=p(0,a,b),\ a\circ b=p(a,b,1)\ \textrm{and}\ \bar{a}=p(1,a,0).
\end{equation}
Then, by Theorem \ref{thm:DM}, $(A,\cdot,\circ,\bar{()},0,1)$ is a de Morgan algebra. Conversely, consider a de Morgan algebra $(A,\cdot,\circ,\bar{()},0,1)$ and define the ternary operation
\begin{equation}\label{p} 
p(a,b,c)=(\bar b\cdot a)\circ (a\cdot c)\circ (b\cdot c)
\end{equation}
then $p$ verifies axioms \ref{C1}--\ref{C5-DM}. We will last prove that \ref{C3} is verified. For the other axioms:
\begin{eqnarray*}
p(0,a,1)&=&(\bar{a}\cdot 0)\circ(a\cdot 1)\circ(0\cdot 1)=0\circ a\circ 0=a\\
p(a,b,a)&=&(\bar{b}\cdot a)\circ(b\cdot a)\circ (a\cdot a)\\
        &=&(\bar{b}\cdot a)\circ(b\cdot a)\circ a=(\bar{b}\cdot a)\circ a=a\\
p(a,0,b)&=&(1\cdot a)\circ(0\cdot b)\circ (a\cdot b)=a\circ (a\cdot b)=a\\
p(a,1,b)&=&(0\cdot a)\circ (1\cdot b)\circ (a\cdot b)=b\circ (a\cdot b)=b\\
p(0,a,b)&=&(\bar a\cdot 0)\circ (0\cdot b)\circ (a\cdot b)=(a\cdot b)=(b\cdot a)=p(0,b,a)\\
p(0,a,a)&=&(\bar a\cdot 0)\circ (0\cdot a)\circ (a\cdot a)=a.
\end{eqnarray*}
Before proving \ref{C3}, note that within a de Morgan algebra, (\ref{p}) can be written as
$$p(a,b,c)= (\bar{b}\circ c)\cdot (b\circ a)\cdot (a\circ c).$$
Using this result and also (\ref{L4}), we have:
\begin{eqnarray*}
p(a,p(b_1,b_2,b_3),c)&=& ([(\bar{b}_2\cdot\bar{b}_1)\circ(b_2\cdot\bar{b}_3)\circ(\bar{b}_1\cdot\bar{b}_3)]\cdot a)\\
                     &&\hspace{-10pt} \circ\,  ([(\bar{b}_2\cdot b_1)\circ(b_2\cdot b_3)\circ(b_1\cdot b_3)]\cdot c)\\
										 &&\hspace{-10pt} \circ\,  (a\cdot c)\\
										 &=& ((\bar{b}_2\cdot\bar{b}_1\cdot a)\circ(b_2\cdot\bar{b}_3\cdot a)\circ(\bar{b}_1\cdot\bar{b}_3\cdot a)\\
                     &&\hspace{-10pt} \circ\,  ((\bar{b}_2\cdot b_1\cdot c)\circ(b_2\cdot b_3\cdot c)\circ(b_1\cdot b_3\cdot c))\\
										 &&\hspace{-10pt} \circ\,  (\bar{b}_2\cdot a\cdot c)\circ (b_2\cdot a\cdot c)\circ (a\cdot c)\\
										 &=& (\bar{b}_2\cdot[(\bar{b}_1\cdot a)\circ(b_1\cdot c)\circ(a\cdot c)])\\
										 &&\hspace{-10pt} \circ\,  (b_2\cdot[(\bar{b}_3\cdot a)\circ (b_3\cdot c)\circ(a\cdot c)])\\
										 &&\hspace{-10pt} \circ\,  (\bar{b}_1\cdot\bar{b}_3\cdot a)\circ (b_1\cdot b_3\cdot c)\circ (a\cdot c)
\end{eqnarray*}
The last line is equal to
\begin{eqnarray*}
p(a,b_1,c)\cdot p(a,b_3,c)&=& ((\bar{b}_1\cdot a)\circ(b_1\cdot c)\circ(a\cdot c))\\
                       &&\hspace{-10pt} \cdot\, ((\bar{b}_3\cdot a)\circ(b_3\cdot c)\circ(a\cdot c))\\
											 &=& (\bar{b}_1\cdot\bar{b}_3\cdot a)\circ(\bar{b}_1\cdot b_3\cdot a\cdot c)\circ(\bar{b}_1\cdot a\cdot c)\\
											 &&\hspace{-10pt} \circ\, (b_1\cdot \bar{b}_3\cdot a\cdot c)\circ(b_1\cdot b_3\cdot c)\circ(b_1\cdot a \cdot c)\\
											 &&\hspace{-10pt} \circ\, (\bar{b}_3\cdot a\cdot c)\circ(b_3\cdot a \cdot c)\circ (a\cdot c)\\
											 &=& (\bar{b}_1\cdot\bar{b}_3\cdot a)\circ(b_1\cdot b_3 \cdot c)\circ(a\cdot c),
\end{eqnarray*}
and consequently, we have that
\begin{eqnarray*}
p(a,p(b_1,b_2,b_3),c)&=& (\bar{b}_2\cdot p(a,b_1,c))\circ(b_2\cdot p(a,b_3,c))\\
										 &&\hspace{-10pt} \circ\,  (p(a,b_1,c)\cdot p(a,b_3,c))\\
										 &=& p(p(a,b_1,c),b_2,p(a,b_3,c)).
\end{eqnarray*}
If starting with a de Morgan algebra $(A,\circ,\cdot,\overline{()},0,1)$ and constructing a system $(A,p,0,1)$ through (\ref{p}) then the de Morgan algebra $(A,\circ',\cdot',\overline{()}',0,1)$ obtained from it is exactly the same as the original one. Indeed:
\begin{eqnarray*}
a\circ' b=p(a,b,1)&=&(\bar b\cdot a) \circ(a\cdot 1)\circ(b\cdot 1)\\
                   &=&(\bar b\cdot a)\circ a\circ b=a\circ b\\
a\cdot' b=p(0,a,b)&=&(\bar a\cdot 0) \circ(0\cdot b)\circ(a\cdot b)=a\cdot b\\
\bar a'=p(1,a,0)&=&(\bar a\cdot 1)\circ (1\cdot 0)\circ(a\cdot 0)=\bar a.
\end{eqnarray*}
If starting with a system $(A,p,0,1)$ verifying \ref{C1}--\ref{C5-DM} and constructing a de Morgan algebra $(A,\circ,\cdot,\overline{()},0,1)$ through (\ref{e-ou-bar}) then the system $(A,p',0,1)$ obtained from it is exactly the same as the original one. Indeed:
\begin{eqnarray*}
p'(a,b,c)&=&(\bar b\cdot a)\circ(a\cdot c)\circ(b\cdot c)\\
         &=&p(p(p(0,\bar b,a),p(0,a,c),1),p(0,b,c),1)\\
									&=& p(a,b,c).
\end{eqnarray*}
Morphisms are defined as usual in a de Morgan algebra. On the ternary side, a morphism $f$ verifies the following requirements:
$$f(0)=0, f(1)=1, f(p(a,b,c))=p(f(a),f(b),f(c)).$$
All the morphisms are trivially preserved by the isomorphism.
\end{proof}
Expression \ref{DMT2} of Theorem \ref{thm:DM} implies that, in a de Morgan algebra, $p(a,a,b)=(\bar a\cdot a)\circ(a\cdot b)\quad$ and $\quad p(a,b,b)=(\bar b\circ b)\cdot(a\circ b)$. This means that in a general de Morgan algebra, the binary operation $p(a,a,b)$ is different from $a\cdot b$ and $p(a,b,b)$ is different from $a\circ b$. These operations are equal when the complement is Boolean in which case the de Morgan algebra is a Boolean algebra. It is worth noting too that, in a de Morgan algebra, the notion of a sum still remains from the ternary structure through $a+b=p(a,b,\bar a)$ and that $(A,+,0)$ is a monoid.
 
\section{Rings and near rings of characteristic two}\label{section:rings}  
Each new interpretation of the ternary operation $p$ satisfying axioms \mbox{\ref{C1}--\ref{C3}} in terms of its derived operations is equivalent to adding a new axiom and gives rise to a new subvariety, as Theorem \ref{thm:2} and Theorem \ref{thm:3} illustrate. An example of a ternary operation obtained from a unitary Abelian near-ring~\cite{near-rings-2021}, which is not necessarily determined by its derived operations, is presented in the next proposition. The algebraic model of the unit interval considered in \mbox{\cite{ccm_magmas, mobi}} is another example.
\begin{proposition}\label{prop}
If $(A,+,\cdot,0,1)$ is a unitary Abelian (right) near-ring, in which $a\cdot 0=0$, then $(A,p,0,1)$ with $p(a,b,c)=a+b(c-a)$ satisfies the axioms \ref{C1} to \ref{C3}.
\end{proposition}
\begin{proof} The proof is straightforward.\end{proof}

When \mbox{$a+b(c-a)$} is changed to \mbox{$(1-b)a+bc$} and $1-b$ is replaced by $1+b$, the formula for $p$, presented in Theorem \ref{thm:2} below, is obtained. 
Recall that a ring (or a near-ring) of characteristic 2 is such that $b+b=0$ for all $b$, so that $1-b=1+b$. Moreover, rewriting $(1+b)a+bc$ as $a+b(a+c)$ gives the formula used in Theorem \ref{thm:3}.
\begin{theorem}\label{thm:2}
Suppose that $(A,p,0,1)$ satisfies axioms \ref{C1} to \ref{C3}. For 
$$\bar{a}=p(1,a,0),\ a\cdot b=p(0,a,b)\ \, \textrm{and}\ \, a+b=p(a,b,\bar a),$$
the following conditions are equivalent:
\begin{enumerate}[label={\rm(\roman*)}]
\item\label{T1a} $(A,+,\cdot,0,1)$ is a unitary ring of characteristic $2$
\item\label{T2a} $p(a,b,c)=(\bar b\cdot a) + (b\cdot c)$
\item\label{T3a} $a\cdot(b+c)=(a\cdot b)+(a\cdot c)$.
\end{enumerate}
\end{theorem}
\begin{proof}
It is clear that \ref{T1a} implies \ref{T3a}. By (\ref{L11}), \ref{T3a} implies $a+a=a\cdot (1+1)=a\cdot0=0$ and hence, considering (\ref{L7}), (\ref{monoid-plus}) and Lemma \ref{lemma:a+a}, \ref{T3a} implies~\ref{T1a}. Moreover, when~$a+a=0$:
\begin{equation}\label{magic}
a+p(a,b,c)=p(a,p(a,b,c),\bar a)\ =\ p(p(a,a,\bar a),b,p(a,c,\bar a))=b\cdot(a+c).
\end{equation}
Consequently, \ref{T3a} implies \ref{T2a} as $p(a,b,c)=a+ba+bc=(1+b)a+bc.$
It remains to prove that \ref{T2a} implies~\ref{T3a}. By \ref{C2}, (\ref{monoid-plus}) and (\ref{L11}), \ref{T2a} implies $1=p(1,a,1)=\bar a+a=1+a+a$ i.e. $a+a=0$. Then (\ref{magic}) and Lemma \ref{lemma:a+a} imply left distributivity:
$$
a\cdot(b+c)=b+p(b,a,c)=b+(1+a)b+ac=ab+ac.
$$
\end{proof}
\noindent The unique non-commutative ring of order $8$, say consisting of all upper triangular binary  $2$-by-$2$ matrices, illustrates Theorem \ref{thm:2}. Note that addition is the Boolean symmetric difference as in a Boolean ring.

\begin{theorem}\label{thm:3}
Suppose that $(A,p,0,1)$ satisfies axioms \ref{C1} to \ref{C3}. For 
$$\bar{a}=p(1,a,0),\ a\cdot b=p(0,a,b)\ \, \textrm{and}\ \, a+b=p(a,b,\bar a),$$
the following conditions are equivalent:
\begin{enumerate}[label={\rm(\roman*)}]
\item\label{T1b} $(A,+,\cdot,0,1)$ is a unitary (right) near ring of characteristic~$2$
\item\label{T2b} $p(a,b,c)=a+(b\cdot(a+c))$
\item\label{T3b} $a+a=0$
\item\label{T4b} $(a+b)\cdot c=(a\cdot c)+(b\cdot c)$.
\end{enumerate}
\end{theorem}
\begin{proof}
It is clear that \ref{T1b} implies \ref{T3b} and, considering Lemma~\ref{lemma:a+a} and associativity of $+$, \ref{T3b} implies~\ref{T1b}. Using (\ref{magic}), condition \ref{T3b} implies \ref{T2b}:
\[
a+p(a,b,c)=b\cdot(a+c)=a+(a+b\cdot(a+c)).
\]
Using (\ref{monoid-plus}) and \ref{C4}, condition \ref{T2b} implies \ref{T3b}: $0=p(a,1,0)=a+a.$
Properties (\ref{L11}) and Lemma~\ref{lemma:a+a} show that \ref{T3b} and \ref{T4b} are equivalent.
\end{proof}
\noindent Several examples of unitary (right) near-rings of characteristic 2 with four elements can be found. The following example illustrates Theorem \ref{thm:3}. Multiplication is neither commutative nor idempotent and addition is the same as Boolean symmetric difference. Note that if the formula $(1+y)x+yz$ is used as the ternary operation $p$ instead of $x+y(x+z)$ then \ref{C3} is not satisfied.
\[
\begin{array}{c|cccc}
\cdot & 0 & u & v & 1 \\ 
\hline 
0  & 0 & 0 & 0 & 0 \\ 
u & 0 & 0 & 0 & u \\ 
v & 0 & u & v & v \\ 
1 & 0 & u & v & 1 \\ 
\end{array} \qquad\quad
\begin{array}{c|cccc}
+ & 0 & u & v & 1 \\ 
\hline 
0  & 0 & u & v & 1 \\ 
u  & u & 0 & 1 & v \\ 
v  & v & 1 & 0 & u \\ 
1  & 1 & v & u & 0 
\end{array}
\] 

\section{MV algebra}\label{section:MV}   

An MV-algebra is a structure of type $(2,1,0)$ that can be defined as follows.
\begin{definition}\label{def:MV}
A MV-algebra is a system $(X,\circ,\ov{()},0)$ such that:
\begin{enumerate}[label={\bf (M\arabic*)}]
\item\label{MV-ass} $x\circ(y\circ z)=(x\circ y)\circ z$
\item\label{MV-unit} $0\circ x=x$
\item\label{MV-inv} $\ov{\ov x}=x$
\item\label{MV-abs} $\ov0\circ x =\ov 0$
\item\label{MV-com} $x\circ\ov{x\circ\ov y}=y\circ\overline{y\circ\ov x}$
\end{enumerate}
\end{definition}
\noindent Usually, the commutativity of the binary operation $\circ$ is included as an axiom of MV-algebras. Nevertheless, Kola\v{r}\'{i}k~\cite{kolarik} proved that commutativity of $\circ$ is a consequence of the other axioms. The canonical example of a MV-algebra consist of the set $X=[0,1]$ and the operations $\bar x=1-x$ and $x\circ y=min(x+y,1)$.

Next, some well-known properties of MV-algebras are presented.
\begin{proposition}\label{properties-MV}
Let $(X,\circ,\overline{()},0)$ be a MV-algebra, and consider the following notations:
\begin{flalign*}
1&=\bar0\\
x\cdot y &=\overline{\bar y\circ\bar x}\\
x\vee y &=x\circ(y\cdot \bar x)\\
x\wedge y&=(\bar y\circ x)\cdot y.
\end{flalign*}
Then, the following properties hold:
\begin{itemize}
\item $x\circ y=y\circ x$ and $x\cdot y=y\cdot x$
\item $x\circ \bar x=1$ and $x\cdot \bar x=0$
\item $x\vee y=\overline{\bar y\wedge\bar x}$
\item $(X,\vee,\wedge,0,1)$ is a distributive lattice
\item $(X,\circ,\cdot,\overline{()},0,1)$ is a Boolean algebra $\Leftrightarrow x\circ y=x\vee y$ $\Leftrightarrow x\circ x=x$
\item $x\circ(y\wedge z)=(x\circ y)\wedge(x\circ z)$ and $x\cdot(y\vee z)=(x\cdot y)\vee(x\cdot z)$.
\end{itemize}
\end{proposition}
\begin{proof}
See for instance \cite{Cig}.
\end{proof}

Considering this properties and Lemma \ref{lemma:complement}, if a MV-algebra comes from a \ref{C1}--\ref{C3} ternary structure, then it is a Boolean algebra. This means that a ternary structure isomorphic to a general MV-algebra cannot contain the full \mbox{\ref{C1}--\ref{C3}} structure. We propose here to replace Axiom \ref{C3} by some of its consequences, namely particular cases of properties (\ref{L5}) and (\ref{L7}). This implies that properties (\ref{L3}) and (\ref{L=}) will not apply in general and, consequently, the operations $\circ$ and $\vee$ will be different. We call the resulting structure a Ternary MV-algebra. 

\begin{definition}\label{definition:ternary-MV}
A Ternary MV-algebra is a system $(A,p,0,1)$ consisting of a set $A$, together with a ternary operation $p:A\times A\times A\to A$ and two constants $0,1\in A$ satisfying:
\begin{enumerate}[label={\bf \quad(T\arabic*)},leftmargin=*,labelindent=0.3em]
\item $p(0,a,1)=a$
\item $p(a,0,b)=a=p(b,1,a)$
\item $p(a,b,a)=a$
\item[{\bf (T4-1)}]$p(a,p(b,c,1),1)=p(p(a,b,1),c,1)$
\item[{\bf (T4-2)}]$p(0,b,c)=\overline{p(\bar c,\bar b,1)}$ and $p(1,b,c)=\overline{p(\bar c,\bar b,0)}$
\item[{\bf (TMV)}]$p(a,b,c)=p(p(\bar b,\bar a,0),p(0,c,b),1)$
\end{enumerate}
\end{definition}

In the next propositions, it is shown that a Ternary MV-algebra is isomorphic to a MV-algebra.
\begin{proposition}\label{thm:MV-1}
Let $(A,p,0,1)$ be a Ternary MV-algebra.
Then, for $\bar{a}=p(1,a,0)$ and $a\circ b=p(a,b,1)$, the structure $(A,\circ,\overline{()},0)$ is a MV-algebra.
\end{proposition}
\begin{proof}
It is clear that $0$ and $1$ are still complements of each other:

$\bar0=p(1,0,0)\ \eq{\ref{C4}}\ 1,\ \bar1=p(1,1,0)\ \eq{\ref{C4}}\ 0$.

Now, we can prove the 5 axioms of Definition \ref{def:MV}:

(M1) $x\circ(y\circ z)=p(x,p(y,z,1),1)\ \,\eq{{\bf (T4-1)}}\ \ p(p(x,y,1),z,1)=(x\circ y)\circ z$

(M2) $0\circ x=p(0,x,1)\ \eq{\ref{C1}}\ x$

(M3) $x\eq{\ref{C1}}p(0,x,1)\ \,\eq{{\bf (T4-2)}}\ \ \overline{p(\bar1,\bar x,\bar0)}\ 
     \eq{\ref{C4}}\ \overline{p(0,\bar x,1)}\ \eq{\ref{C1}}\ \overline{\overline{x}}$

(M4) $1\circ x=p(1,x,1)\ \eq{\ref{C2}}\,1$
\begin{flalign*}
(\textrm{M}5)\,x\circ\overline{(x\circ\bar y)}&=p(x,\overline{p(x,\bar y,1)},1)\
\eq{{\bf(T4-2)}}\ \,p(x,p(0,y,\bar x),1)\\
&\eq{\ref{C4}}\ p(p(x,\bar1,0),p(0,y,\bar x),1)\ \
\eq{{\bf (TMV)}}\ \ p(1,\bar x,y)\\ 
&\eq{{\bf(T4-2)}}\ \ \overline{p(\bar y,x,0)}\ \ \,
\eq{{\bf (TMV)}}\ \ \,\overline{p(p(\bar x,y,0),p(0,0,x),1)}\qquad\quad\\
&\eq{\ref{C4}}\ \overline{p(\bar x,y,0)}\ \
\eq{{\bf(T4-2)}}\ \,p(1,\bar y,x)\\
&=y\circ\overline{(y\circ\bar x)}.
\end{flalign*}
\end{proof}

\begin{proposition}\label{thm:MV-2}
Let $(A,\circ,\overline{()},0)$ be a MV-algebra, and consider the usual dual operation $a\cdot b=\overline{\bar b\circ\bar a}$. Then, for 
\begin{equation}\label{p-MV}
p(a,b,c)=((\bar a\circ\bar b)\cdot a)\circ(b\cdot c),
\end{equation}
the structure $(A,p,0,1)$ is a Ternary MV-algebra.
\end{proposition}
\begin{proof}
First, from~(\ref{p-MV}) and using Proposition \ref{properties-MV}, we observe the following correspondences.
\begin{enumerate}[label={},leftmargin=*,labelindent=0.3em]
\item $p(1,a,0)=\bar a$
\item $p(0,a,b)=0\circ(a\cdot b)=a\cdot b$
\item $p(b,\bar a,0)=(\bar b\circ a)\cdot b=a\wedge b$
\item $p(1,\bar b,a)=b\circ(\bar b\cdot a)=b\circ(a\cdot \bar b)=b\vee a=a\vee b$
\item $p(a,b,1)=((\bar a\circ\bar b)\cdot a)\circ b=(\bar b\wedge a)\circ b=(\bar b\circ b)\wedge (a\circ b)=a\circ b.$
\end{enumerate}
Note that these results are compatible with definitions (\ref{def:bar}) to (\ref{def:vee}).
Now, we can prove the properties of $p$ included in Definition \ref{definition:ternary-MV}:
\begin{flalign*}
(\textrm{T}1)\, &p(0,a,1)=0\circ(a\cdot1)=a\\
(\textrm{T}2)\, &p(a,0,b)=(1\cdot a)\circ0=a\ \textrm{and}\ p(a,1,b)=(\bar a\cdot a)\circ b=0\circ b=b
\end{flalign*}
\begin{flalign*}
(\textrm{T}3)\, p(a,b,a)&=((\bar a\circ\bar b)\cdot a)\circ(b\cdot a)=(\overline{b\cdot a}\cdot a)\circ(b\cdot a)=(b\cdot a)\vee a\\
                        &=(b\cdot a)\vee (1\cdot a)=(b\vee1)\cdot a=1\cdot a=a
\end{flalign*}
\begin{flalign*}
(\textrm{T4-1})\, p(a,p(b,c,1),1)&=a\circ(b\circ c)=(a\circ b)\circ c=p(p(a,b,1),c,1)\quad
\end{flalign*}
\begin{enumerate}[leftmargin=3.5em]
\item[(T4-2)]$p(0,b,c)=b\cdot c=\overline{\bar c\circ\bar b}=\overline{p(\bar c,\bar b,1)}$
\item[]$p(1,b,c)=c\vee \bar b=\overline{b\wedge\bar c}=\overline{p(\bar c,\bar b,0)}$
\end{enumerate}
\begin{enumerate}[leftmargin=3.5em]
\item[(TMV)]$p(a,b,c)=((\bar a\circ\bar b)\cdot a)\circ(b\cdot c)=(\bar b\wedge a)\circ(c\cdot b)=(a\wedge \bar b)\circ(c\cdot b)$
\item[]\qquad\quad\ \, $=p(p(\bar b,\bar a,0),p(0,c,b),1)$.
\end{enumerate}
\end{proof}
\begin{theorem}\label{thm:MV-3}
Ternary MV-algebras and MV-algebras are isomorphic.
\end{theorem}
\begin{proof}
Let $(A,p,0,1)$ be a Ternary MV-algebra and consider the operations $\bar{a}=p(1,a,0)$ and $a\circ b=p(a,b,1)$. Then, by Proposition~\ref{thm:MV-1}, the structure $(A,\circ,\overline{()},0)$ is a MV-algebra. By Proposition \ref{thm:MV-2}, a new Ternary MV-algebra $(A,p',0,1)$ is recovered. We prove now that the new Ternary MV-algebra is equal to the original one i.e. that $p'=p$:
\begin{flalign*}
p'(a,b,c)&=((\bar a\circ\bar b)\cdot a)\circ(b\cdot c)=(\bar b\wedge a)\circ(b\cdot c)\\
         &=(a\wedge \bar b)\circ(c\cdot b)=p(p(\bar b,\bar a,0),p(0,c,b),1)=p(a,b,c).
\end{flalign*}
Conversely, let us begin with a MV-algebra $(A,\circ,\overline{()},0)$. Then, using Proposition~\ref{thm:MV-2}, a Ternary MV-algebra $(A,p,0,1)$ is obtain with 
$p(a,b,c)=((\bar a\circ\bar b)\cdot a)\circ(b\cdot c)$ which, by Proposition \ref{thm:MV-1}, gives back a MV-algebra $(A,\circ',\overline{()}',0)$. We prove now that $\circ'=\circ$ and $\overline{()}'=\overline{()}$:
\begin{flalign*}
a\circ'b &=p(a,b,1)\\
         &=((\bar a\circ\bar b)\cdot a)\circ b=(\bar b\wedge a)\circ b=(\bar b\circ b)\wedge (a\circ b)=a\circ b.\\
	\bar a'&=p(1,a,0)=((0\circ\bar a)\cdot 1)\circ(a\cdot 0)=\bar a.
\end{flalign*}
\end{proof}

To compare with equalities between binary operations observed in Boolean and de Morgan algebras, let us notice that, in a ternary MV-algebra, the following relations hold:
$$ a\wedge b=p(a,\bar b,b)\quad \textrm{and}\quad b\vee a=p(a,\bar a,b).$$
\noindent Note also that in a Ternary MV-algebra, the notion of a sum is still well defined through $a+b=p(a,b,\bar a)$, with $a+0=a=0+a$ and $a+1=\bar a=1+a$.

\section{conclusion}\label{section:conclusion}   
We have presented examples of ternary structures that provide a common background for several classical algebras. It has long been recognized that De Morgan algebras with Boolean complements or MV-algebras with idempotency are Boolean algebras. However, other characteristics of Boolean algebras, like the existence of a sum leading to Boolean rings are not so easily generalized (see \cite{Chajda} for an example). Here, the sum defined as $x+y=p(x,y,\bar x)$ is well defined in all
structures derived from a ternary system and allows for a generalization of the notion of a ring. With respect to the algebras considered here, we observe the following table:
\[
\begin{array}{c|c|c|c|c|c|}
 & T1-T3 & T4 & BC^{\dagger} & a\cdot a=a & LD^{*}\\ 
\hline 
Boolean\ algebra        & \checkmark & \checkmark & \checkmark & \checkmark & \checkmark\\ 
\hline 
De\ Morgan\ algebra     & \checkmark & \checkmark & \times     & \checkmark & \times\\ 
\hline 
MV-algebra          & \checkmark & \times     & \checkmark & \times     & \times\\ 
\hline 
Ring, char=2        & \checkmark & \checkmark & \times     & \times     & \checkmark\\ 
\hline 
Near\ ring, char=2  & \checkmark & \checkmark & \times     & \times     & \times\\ 
\hline 
\end{array}
\]
$^{\dagger}$ BC means Boolean complement: $(1+a)\cdot a=0$.\newline
$^{*}$ LD means left distributivity of $\cdot$ over $+$: $a\cdot(b+c)=(a\cdot b)+(a\cdot c).$\newline
Much more can be done in the study of structures involving the operation $+$.

In Section \ref{section:MV}, we explained how the initial ternary structure \ref{C1}--\ref{C3} is generalized to allow for MV-algebras. Of course, the structure can be modified in other ways given other type of structures. With only \ref{C1}--\ref{C2}, the unary operation $\bar x=p(1,x,0)$ is not an involution which could be a starting point to study a ternary version of Heyting algebras, for example.



\begin{thebibliography}{999}
\bibitem{Bandelt} H-S. Bandelt, J Hedl\'{i}cov\'{a}, \emph{Median algebras}, Discrete Mathematics \textbf{45} (1983) 1--30.


\bibitem{Belnap} N. Belnap, \emph{A useful four-valued logic}, J.M. Dunn, G. Epstein (Eds.), Modern Uses of Multi-Valued Logic, Reidel, Dordrecht (1977) 8-37.

\bibitem{Ponse} J. A. Bergstra, A. Ponse, \emph{Proposition algebra}, ACM Trans. Comput. Logic \textbf{12} (3) (2011) No.21, 1--36.

\bibitem{Birk} G. Birkhoff, S. A. Kiss, \emph{A ternary operation in distributive lattices}, Bull. Amer. Math. Soc. \textbf{53} (8) (1947) 749--752.


\bibitem{Birk-livro} G. Birkhoff, \emph{Lattice Theory}, Amer. Math. Soc. Colloquium Publications, Vol. \textbf{25}, rev. ed. (1948).




\bibitem{Chajda} I. Chajda, H. L{\"a}nger, \emph{Ring-Like Structures Corresponding to MV-algebras via Symmetric Difference}, Sitzungsber. Abt. II \textbf{213} (2004) 33--41. 

\bibitem{ChurchPM} A. Church, \emph{Conditioned disjunction as a primitive connective for the prepositional calculus}, Portugaliae Mathematica, vol. \textbf{7} (1948) 87--90.

\bibitem{Cig} R.L.O. Cignoli, I.M.L. D'Ottaviano, D. Mundici, \emph{Algebraic foundations of many-valued reasoning}, Trends in Logic (2000).


\bibitem{Salibra2} K. Cvetko-Vah, A. Salibra, \emph{The connection of skew Boolean algebras and discriminator varieties to Church algebras}, Algebra Universalis \textbf{73} (2015) 369--390.

\bibitem{menger} W. A. Dudek, V. S. Trokhimenko, \emph{Algebra of Multiplace Functions}, De Gruyter, Berlin, 2012.

\bibitem{ccm_magmas} J. P. Fatelo, N. Martins-Ferreira, \emph{Internal monoids and groups in the category of commutative cancellative medial magmas}, Portugaliae Mathematica, Vol. \textbf{73}, Fasc. 3 (2016) 219--245. 

\bibitem{mobi} J. P. Fatelo, N. Martins-Ferreira, \emph{Mobi algebra as an abstraction to the unit interval and its comparison to rings}, Communications in Algebra \textbf{47} (3) (2019) 1197--1214. 

\bibitem{preprint-TBA} J. P. Fatelo, N. Martins-Ferreira, \emph{A new look at ternary Boolean algebras}, arXiv:\-2109.06259. 

\bibitem{preprint-TBA2} J. P. Fatelo, N. Martins-Ferreira, \emph{A refinement of ternary Boolean algebras}, arXiv:\-2203.08012.

\bibitem{affine} J. P. Fatelo, N. Martins-Ferreira, \emph{Affine mobi spaces}, Bollettino dell’Unione Matematica Italiana \textbf{15} (2022) 589--604.

\bibitem{mobi2sphere} J. P. Fatelo, N. Martins-Ferreira, \emph{Mobi spaces and geodesics for the \mbox{N-sphere}}, Cah. Topol. G\'{e}om. Diff\'{e}r. Cat\'{e}g. \textbf{63} (1) (2022) 59--88.

\bibitem{Font} J. M. Font, \emph{Belnap's Four-Valued Logic and De Morgan Lattices}, Logic Journal of IGPL \textbf{5} (3) (1997) 1-29

\bibitem{Grau} A. A. Grau, \emph{Ternary Boolean algebras}, Bull. Amer. Math. Soc. \textbf{53} (6) (1947) 567--572.

\bibitem{hoare} C. A. R. Hoare, \emph{A couple of novelties in the propositional calculus}, Z. Math. Logik Grundlag. Math. \textbf{31} (2) (1985) 173--178.

\bibitem{Isbell} J. R. Isbell, \emph{Median Algebra}, Transactions of the American Mathematical Society \textbf{260} (2) (1980) 319--362.



\bibitem{Kalman} J. A. Kalman, \emph{Lattices with involution},  Transactions of the American Mathematical Society, \textbf{87} (1958) 485–91.





Discrete and Applied Mathematics \textbf{2}  (2) (2019) $\#$P2.01.


\bibitem{near-rings-2021} R. Lockhart, \emph{The theory of Near-Rings}, Lecture Notes in Mathematics, volume 2295 (2021), Springer.

\bibitem{kolarik} M. Kola\v{r}\'{i}k, \emph{Independence of the axiomatic systems for a MV-algebras}, Math. Slovaca \textbf{63} (2013) 1--4.


\bibitem{NMF2012} N.~Martins-Ferreira, \emph{On distributive lattices and Weakly {M}al'tsev categories}, J. Pure Appl. Algebra,  \textbf{216} (2012) 1961--1963.

\bibitem{Pad} R. Padmanabhan, W. McCune, \emph{Single identities for ternary Boolean algebra}, Comput. Math. Appl. \textbf{29} (2) (1995) 13--16.


\bibitem{Sholander} M. Sholander, \emph{Trees, lattice, order and betweenness}, Proc. Amer. Math. Soc. \textbf{3} (1952) 369--381.

\bibitem {Salibra1} A. Salibra, A. Ledda, F. Paoli, T. Kowalski, \emph{Boolean-like algebras}, Algebra Universalis \textbf{69} (2013) 113--138.


\end{thebibliography}
\end{document}